\documentclass[12pt,reqno,a4paper]{amsart}

\usepackage{graphicx}
\usepackage{transparent}

\usepackage[latin1]{inputenc}  

\usepackage{mathrsfs}
\usepackage{amssymb,amsmath, amsfonts, amsthm}
\usepackage{latexsym}
\usepackage[dvips]{epsfig}

\newtheorem{theorem}{Theorem}[section]
\newtheorem{proposition}[theorem]{Proposition}
\newtheorem{lemma}[theorem]{Lemma}
\newtheorem{corollary}[theorem]{Corollary}
\newtheorem{remark}[theorem]{Remark}



\usepackage{epsfig,color,xcolor}

\newcommand{\ble}{\begin{lemma}}
\newcommand{\ele}{\end{lemma}}
\newcommand{\be}{\begin{equation*}}
\newcommand{\ee}{\end{equation*}}

\newcommand{\bel}{\begin{equation}}
\newcommand{\eel}{\end{equation}}

\newcommand{\fr}{\frac }

\newcommand{\lap}{\Delta}
\newcommand{\N}{\mathbb{N}}

\newcommand{\R}{\mathbb{R}}

\renewcommand{\to}{\rightarrow}
\newcommand{\To}{\longrightarrow}

\newcommand{\xip}{x_{i,p}}
\newcommand{\xjp}{x_{j,p}}

\newcommand{\mip}{\mu_{i,p}}

\newcommand{\upp}{u_p}

\def\sideremark#1{\ifvmode\leavevmode\fi\vadjust{\vbox to0pt{\vss
 \hbox to 0pt{\hskip\hsize\hskip1em
 \vbox{\hsize2.1cm\tiny\raggedright\pretolerance10000
  \noindent #1\hfill}\hss}\vbox to15pt{\vfil}\vss}}}%

\begin{document}

\numberwithin{equation}{section}
\parindent=0pt
\hfuzz=2pt
\frenchspacing

\title[]{Blow up of solutions of semilinear heat equations in non radial domains of $\R^2$}

\author[]{Francesca De Marchis, Isabella Ianni}

\address{Francesca De Marchis, Dipartimento di Matematica,  Universit\`a  di Roma {\em Tor Vergata}, Via della Ricerca
Scientifica 1, 00133 Roma, Italy}
\address{Isabella Ianni, Dipartimento di Matematica e Fisica, Seconda Universit\`a di Napoli, V.le Lincoln 5, 81100 Caserta, Italy}

\thanks{2010 \textit{Mathematics Subject classification:} 35K91, 35B35, 35B44, 35J91  }

\thanks{ \textit{Keywords}: semilinear heat equation, finite-time blow-up,  sign-changing stationary solutions, asymptotic behavior.}

\thanks{Research partially supported by FIRB project: {\sl Analysis and Beyond} and PRIN project 2012: 74FYK7-005. }

\begin{abstract} We consider the semilinear heat equation
\begin{equation}\label{problemAbstract}\left\{\begin{array}{ll}v_t-\Delta v= |v|^{p-1}v & \mbox{ in }\Omega\times (0,T)\\
v=0 & \mbox{ on }\partial \Omega\times (0,T)\\
v(0)=v_0 & \mbox{ in }\Omega
\end{array}\right.\tag{$\mathcal P_p$}
\end{equation}
where $p>1$, $\Omega$ is a smooth bounded domain of $\R^2$, $T\in (0,+\infty]$ and $v_0$ belongs to a suitable space.
We give general conditions for a family $u_p$ of sign-changing  stationary solutions of \eqref{problemAbstract}, under which the solution of \eqref{problemAbstract} with initial value $v_0=\lambda u_p$ blows up in finite time if $|\lambda-1|>0$ is sufficiently small and $p$ is sufficiently large. 
Since for $\lambda=1$ the solution is global, this shows that, in general, the set of the initial conditions for which the solution is global is not star-shaped with respect to the origin. In \cite{DicksteinPacellaSciunzi} this phenomenon has been previously observed  in the case when the domain is a ball and the sign changing stationary solution is radially symmetric. Our conditions are more general and we provide examples of stationary solutions $u_p$ which are not radial and exhibit the same behavior.
\end{abstract}

\maketitle

%

\section{Introduction}
We consider the nonlinear heat equation
\begin{equation}\label{problem}\left\{\begin{array}{ll}v_t-\Delta v= |v|^{p-1}v & \mbox{ in }\Omega\times (0,T)\\
v=0 & \mbox{ on }\partial \Omega\times (0,T)\\
v(0)=v_0 & \mbox{ in }\Omega
\end{array}\right.
\end{equation}
where $\Omega\subset \R^N$, $N\in\N$, $N\geq 1$ is a smooth bounded domain, $p>1$, $T\in (0, +\infty]$ and $v_0\in C_0(\Omega)$, where
\[C_0(\Omega)=\{v\in C(\bar\Omega),\ v=0  \mbox{ on } \partial\Omega\}.\]
It is well known that the initial value problem \eqref{problem} is locally well-posed in $C_0(\Omega)$ and admits both nontrivial global solutions and blow-up solutions. Denoting with $T_{v_0}$ the maximal existence time of the solution of \eqref{problem}, we define the set of initial conditions for which the corresponding solution is global, i.e.
\[\mathcal{G}=\{v_0\in C_0(\Omega),\ T_{v_0}=+\infty\}\]
and its complementary set of initial conditions for which the corresponding solution blows-up in finite time:
\[\mathcal{B}=\{v_0\in C_0(\Omega),\ T_{v_0}<+\infty\}.\]

For a fixed  $w\in C_0(\Omega)$, $w\neq 0$ and $v_0=\lambda w$, $\lambda\in\R$, it is known that if $|\lambda|$ is small enough then $v_0\in\mathcal{G}$, while if $|\lambda|$ is sufficiently large then $v_0\in\mathcal{B}$. \\

Moreover, for any $N\geq 1$, considering nonnegative initial data, it can be easily proved that the set $\mathcal G^+=\{v_0\in\mathcal{G},\ v_0\geq 0\}$ is star-shaped with respect to $0$ (indeed it is convex, see \cite{quittnersouplet}). On the contrary when the initial condition changes sign $\mathcal G$ may not be star-shaped.\\

Indeed for $N\geq 3$, in \cite{CazenaveDicksteinWeissler} and \cite{MarinoPacellaSciunzi} it has been shown that there 
exists  $p^*<p_S$, with $p_S=\frac{N+2}{N-2}$, such that $\forall \ p \in \ (p^*,p_S)$
the elliptic problem
\begin{equation}\label{stationaryProblem}
\left\{\begin{array}{ll}-\Delta u_p= |u_p|^{p-1}u_p & \mbox{ in }\Omega\\
u_p=0 & \mbox{ on }\partial \Omega,
\end{array}\right.
\end{equation}
admits a sign a changing solution $u_p$ for which there exists $\epsilon>0$ such that if $0<|1-\lambda|<\epsilon$ then $\lambda u_p\in\mathcal{B}$.
More precisely  this result has been proved first in \cite{CazenaveDicksteinWeissler}  when $\Omega$ is the unit ball and $u_p$ is any sign-changing radial solution of \eqref{stationaryProblem}, and then   
in  \cite{MarinoPacellaSciunzi} for general domains $\Omega$ and sign-changing solutions $u_p$ of \eqref{stationaryProblem} (assuming w.l.g. that $\|u_p^+\|_{\infty}=\|u_p\|_{\infty}$), satisfying the following conditions
\[\tag{a}\int_{\Omega}|\nabla u_p|^2 dx\rightarrow 2S^{\frac{N}{2}}\]
\[\tag{b}\left|\frac{\min u_p}{\max u_p}\right|\rightarrow 0\]
as $p\rightarrow p^*$, where $S$ is the best Sobolev constant for the embedding of $H^1_0(\Omega)$ into $L^{2^*}(\Omega)$.
When $N=1$ such a result does not hold since it is easy to see that for any sign changing solution $u_p$ of \eqref{stationaryProblem}, $\lambda u_p\in \mathcal G$ for $|\lambda|<1$ and $\lambda u_p\in \mathcal B$ for $|\lambda|>1$.

\

The case $N=2$ was left open in the papers \cite{CazenaveDicksteinWeissler} and \cite{MarinoPacellaSciunzi}, mainly because there is not a critical Sobolev exponent in this dimension so that the conditions and results of these papers are meaningless. Recently inspired by \cite{GrossiGrumiauPacella1}, \cite{GrossiGrumiauPacella2}, Dickstein, Pacella and Sciunzi in \cite{DicksteinPacellaSciunzi} succeeded
 to prove a blow up theorem similar to the one in \cite{CazenaveDicksteinWeissler}, considering radial sign changing stationary solutions $u_p$ of \eqref{stationaryProblem} in the unit ball for large exponents $p$. Indeed, the condition $p\rightarrow +\infty$ in dimension $N=2$ can be considered to be the natural extension of the condition $p\rightarrow p_S$ for $N\geq 3$.

In this paper we consider again the case $N=2$ but the bounded domain $\Omega$ is not necessarily a ball. We deal with sign-changing solutions $u_p$ of \eqref{stationaryProblem}  with the following two properties:
\begin{equation} \tag{A}
\exists C>0,  \mbox{ such that }\ \ p\int_{\Omega} |\nabla u_p|^{2}dx \leq C,
\end{equation}

\begin{equation}  \tag{B}
\lim_{R\rightarrow +\infty}\lim_{p\rightarrow +\infty} \mathcal S_{p,R}=0,
\end{equation}
where, for $R>0$, 
\[
\mathcal S_{p,R}:=\sup \left\{ \left|\frac{u_p(y)}{u_p(x_p^+)}\right|^{p-1}\ : \ y\in\Omega, \ |y-x_p^+|> R\mu_p^+\right\},
\]
for $x_p^+$ such that $|u_p(x_p^+)|=\|u_p\|_{\infty}$ and  $\mu_p^+:=\frac{1}{\sqrt{p|u_p(x_p^+)|^{p-1}}}$.

\

Our main result is the following 
\begin{theorem}\label{teorema:principale}
Let $N=2$ and $u_p$ be a family of sign-changing solutions of problem \eqref{stationaryProblem} satisfying $(A)$ and $(B)$. Then, up to a subsequence, there exists $p^*>1$ such that for $p>p^*$ there exists $\epsilon=\epsilon (p) >0$ such that if $0<|1-\lambda|<\epsilon$, then 
\[\lambda u_p\in\mathcal B.\]
\end{theorem}

\

A few comments on conditions $(A)$ and $(B)$ are needed. The first one is an estimate of the asymptotic behavior, as $p\rightarrow +\infty$, of the energy of the solutions. It is satisfied by different kinds of sign changing solutions (see \cite{DeMarchisIanniPacella1}, \cite{DeMarchisIanniPacella2}, \cite{GrossiGrumiauPacella1}), in particular by the radial ones in the ball (see \cite{DicksteinPacellaSciunzi}, \cite{GrossiGrumiauPacella2}).
The condition $(B)$ is more peculiar
and essentially concerns the asymptotic behavior of $|u_p(y)|^{p-1}$ for points $y$ which are not too close to $x_p^+$; note that $p|u_p(y)|^{p-1}$ can also be divergent since $\liminf_{p\rightarrow +\infty}\|u_p\|_{\infty}\geq 1$.
It is satisfied, in particular, by sign changing radial solutions $u_{p,\mathcal K}$ of \eqref{stationaryProblem} having any fixed number $\mathcal K$ of nodal regions (see Section \ref{sectionCasoRadiale} for details).
But it also holds for a class of sign changing solutions in more general domains as we show in the next theorem.

\

\begin{theorem}\label{teoremaExistenceAB}
Let $\Omega\subset\R^2$ be a 
smooth bounded  domain, containing the origin,  invariant under the action of a cyclic group $G$ of
rotations about the origin
with 
\begin{equation} \label{assumptionGruppoSimmetria}
|G|\geq me
\end{equation}
 ($|G|$ is the order of $G$), for a certain $m>0$.\\
Let $(\upp)$ be a family of sign changing $G$-symmetric   solutions of \eqref{stationaryProblem} 
such that 
\begin{equation} \label{energiaLimitata}
p\int_{\Omega}|\nabla\upp|^2 dx\leq\alpha\,8\pi e,\ \mbox{ for $p$ large}
\end{equation}
and for some  $\alpha<m+1$. Then $u_p$
 satisfies $(A)$ and $(B)$ up to a subsequence.
\end{theorem}

\

\begin{remark}The existence of sign-changing solutions of \eqref{stationaryProblem} satisfying the assumption \eqref{energiaLimitata}  of Theorem \ref{teoremaExistenceAB} has been proved  in \cite{DeMarchisIanniPacella1} for  $m\geq 4$ and for large $p$.
\end{remark}

\

Putting together Theorem \ref{teorema:principale} and Theorem \ref{teoremaExistenceAB} one gets the extension of the blow-up result of \cite{DicksteinPacellaSciunzi} to other symmetric domains. Note that this in particular includes the case of sign changing radial solutions in a ball, providing so a different proof of the result of \cite{DicksteinPacellaSciunzi} which strongly relied on the radial symmetry.

\

The proof of Theorem \ref{teorema:principale} follows the  same strategy as the analogous results in \cite{CazenaveDicksteinWeissler, MarinoPacellaSciunzi, DicksteinPacellaSciunzi} being a consequence of the following proposition, which is  a particular case of \cite[Theorem 2.3]{CazenaveDicksteinWeissler}
\begin{proposition} \label{proposition:generaleCazenaveDickWeiss} 
Let $u$ be a sign changing solution of \eqref{stationaryProblem}
and let $\varphi_1$ be a positive eigenvector of the self-adjoint operator $L$ given by $L\varphi=-\Delta \varphi-p|u|^{p-1}\varphi$, for $\varphi\in H^2(\Omega)\cap H^1_0(\Omega)$. Assume that
\begin{equation}\label{condizioneProdottoScalareDiversoDaZero} \int_{\Omega}u\varphi_1\neq 0.
\end{equation}
Then there exists $\epsilon >0$ such that if $0<|1-\lambda|<\epsilon$, then $\lambda u\in\mathcal{B}$.
\end{proposition}

More precisely we will show that under the assumptions $(A)$ and $(B)$  condition \eqref{condizioneProdottoScalareDiversoDaZero} is satisfied for $p$ large (see Theorem \ref{condizioneProdottoScalarePositivo} in the following). This proof is based on rescaling arguments about the maximum point of $u_p$: using the properties of the solution of  the limit problem, we analyze the asymptotic behavior as $p\rightarrow +\infty$ of the rescaled solutions and of the rescaled first eigenfunction of the linearized operator at $u_p$. In this analysis the assumption $(B)$ is crucial.

\

\

To get Theorem \ref{teoremaExistenceAB} we prove a more general result which shows that  condition $(B)$ holds for  a quite general class of solutions $u_p$ of \eqref{stationaryProblem} (see Lemma \ref{condizioneNecessariaGeneralissima} and Theorem \ref{corollaryCheUso}). 

\

\

The paper is organized as follows.
\\
In Section \ref{sectionPreliminary} we collect some properties satisfied by any  family of solutions $u_p$ under condition $(A)$ and we give a characterization of  condition $(B)$.
\\
In Section \ref{sectionAsymptoticAnaysis} we carry out an asymptotic spectral analysis under assumption $(A)$ and $(B)$, studying the asymptotic behavior of the first eigenvalue and of the first eigenfunction of the linearized operator at $u_p$.
\\  
Section \ref{sectionProofTheorem} is devoted to the proof of Theorem \ref{teorema:principale} via rescaling arguments.
\\
In Section \ref{sectionSufficentConditionToB} we find a  sufficient condition which ensures the validity of  property $(B)$ (see Lemma \ref{condizioneNecessariaGeneralissima}) and we select a general class of solutions to \eqref{stationaryProblem} which satisfies this sufficient condition (see Theorem \ref{corollaryCheUso}).
\\
In Section \ref{sectionCasoGSimmetrico} we prove Theorem \ref{teoremaExistenceAB}.

Finally in Section \ref{sectionCasoRadiale} we show that also the sign-changing radial solutions in the unit ball satisfy the assumptions of Theorem \ref{teoremaExistenceAB}.

\

\

\section{Preliminary results} 
\label{sectionPreliminary}

We fix some notation. For a given family $(u_p)$ of sign-changing stationary solutions of \eqref{problem} we denote by
\begin{itemize}
\item $NL_p$ the nodal line of $\upp$,
\item $x_p^{\pm}$ a maximum/minimum point in $\Omega$ of $u_p$, i.e. $u_p(x_p^{\pm})=\pm\|u_p^{\pm}\|_{\infty}$,
\item $\mathcal{N}_p^{\pm}:=\{x\in\Omega\ :\ u_p^{\pm}(x)\neq 0\}\subset\Omega$ denote the positive/negative domain of $u_p$,
\item $\mu_p^{\pm}:=\frac{1}{\sqrt{p|u_p(x_p^{\pm})|^{p-1}}}$,
\item  $\widetilde{\Omega}_p:=\widetilde{\Omega}_p^{+}=\frac{\Omega-x_p^{+}}{\mu_p^{+}}=\{x\in\mathbb R^2: x_p^{+}+\mu_p^{+}x\in \Omega\}$,
\item  $\widetilde{\mathcal N_p^+}:=\frac{\mathcal N_p^+-x_p^{+}}{\mu_p^{+}}=\{x\in\mathbb R^2: x_p^{+}+\mu_p^{+}x\in \mathcal N_p^+\}$,
\item $d(x,A):=dist (x,A)$, for any $x\in\mathbb R^2, A\subset\mathbb R^2$.
\end{itemize}

We  assume w.l.o.g. that $\|u_p\|_{\infty}=\|u_p^+\|_{\infty}$.

\

In the next two lemmas we collect some useful properties  holding under condition $(A)$.

\

\begin{lemma}\label{primoLemmaCheVieneDaA}Let $(\upp)$ be a family of solutions to \eqref{stationaryProblem} and assume that $(A)$ holds. Then
\[\|u_p\|_{L^{\infty}(\Omega)}\leq C\]
\begin{equation}\label{maxBoundedAwayFromZero}
\liminf_{p\rightarrow +\infty} u_p(x_p^{\pm})\geq 1
\end{equation} 
\begin{equation}\label{muVaAZero}\mu_p^{\pm}\rightarrow 0 \ \mbox{ as }\ p\rightarrow +\infty\end{equation}
\end{lemma}
\begin{proof} It is well known, see for instance \cite[Lemma 2.1]{DeMarchisIanniPacella2} or \cite{GrossiGrumiauPacella1}.
\end{proof}

\begin{lemma} \label{convergenzaVp}
Let $(\upp)$ be a family of solutions to \eqref{stationaryProblem} and assume that $(A)$ holds. Then, up to a subsequence, the rescaled function
\begin{equation}
v_p^+(x):=\fr{p}{\upp(x_p^+)}(\upp(x_p^++\mu_p^+ x)-\upp(x_p^+))
\end{equation}
defined on $\widetilde{\Omega}_{p}$ converges to $U$ in $C^1_{loc}(\mathbb R^2)$, where
\begin{equation}\label{v0}
U(x)=\log\left(\fr1{1+\fr18 |x|^2}\right)^2
\end{equation}
is the solution of the Liouville problem
\begin{equation}\label{LiouvilleProblem}
\left\{
\begin{array}{lr}
-\lap U=e^{U}\ \mbox{ in }\ \R^2\\
U\leq 0, \ U(0)=0\ \mbox{ and }\ \int_{\mathbb{R}^2}e^{U}=8\pi.\end{array}
\right.
\end{equation}
Moreover
\begin{equation}\label{xppiuNonVedeLineaNodale}
\frac{d(x_p^+, NL_p)}{\mu_p^{+}}\rightarrow +\infty\ \mbox{ as }\ p\rightarrow +\infty
\end{equation}
and \begin{equation}\label{xppiuNonVedeboundary}
\frac{d(x_p^+, \partial\Omega)}{\mu_p^{+}}\rightarrow +\infty\ \mbox{ as }\ p\rightarrow +\infty.
\end{equation}
\end{lemma}
\begin{proof}
It is well known, see for instance \cite[Proposition 2.2 \& Corollary 2.4]{DeMarchisIanniPacella2}
\end{proof}

\

\

Observe that, under condition $(A)$, by \eqref{xppiuNonVedeboundary}, for any $R>0$ there exists $p_R>1$ such that the set $\{y\in\Omega, \ |y-x_p^+|> R\mu_p^+ \}\neq \emptyset$ for $p\geq p_R$. As a consequence for any $R>0$ the value  \[
\mathcal S_{p,R}:=\sup \left\{ \left|\frac{u_p(y)}{u_p(x_p^+)}\right|^{p-1}\ : \ y\in\Omega, \ |y-x_p^+|> R\mu_p^+\right\}
\]
in the definition of condition $(B)$ is well-defined for $p\geq p_R$.\\

Next we give a characterization of condition $(B)$:
\begin{proposition}\label{proposition:caratterizzazioneB}
Assume that $u_p$ is a family of sign-changing solutions to \eqref{stationaryProblem} which satisfies condition $(A)$.\\ Then for any $R>0$ there exists $p_R>1$ such that the set $\{y\in\mathcal N_p^+, \ |y-x_p^+|> R\mu_p^+ \}\neq \emptyset$ for $p\geq p_R$ and so
 \begin{equation} \label{definitionMpR}
\mathcal M_{p,R}:=\sup \left\{ \left|\frac{u_p(y)}{u_p(x_p^+)}\right|^{p-1}\ : \ y\in\mathcal N_p^+, \ |y-x_p^+|> R\mu_p^+\right\}
\end{equation} 
is well defined.\\
Moreover condition $(B)$ is equivalent to 
\[
 \left\{
 \begin{array}{lr}
 \displaystyle{ \lim_{p\rightarrow +\infty}}\frac{\|u_p^-\|_{L^{\infty}(\Omega)}^{p-1}}{\|u_p^+\|_{L^{\infty}(\Omega)}^{p-1}}= 0  & \qquad (B1)\\
 \\
 \displaystyle{\lim_{R\rightarrow +\infty}}\lim_{p\rightarrow +\infty}  \mathcal M_{p,R}=0 & \qquad (B2)
 \end{array}
 \right.
 \]
 \end{proposition}
\begin{proof}
For $R>0$ and $p>1$ define the set \[{\Omega}_{p,R}:=\left\{y\in\Omega, \ |y-x_p^+|> R\mu_p^+\right\}\ (\subseteq \Omega).\]
In order to prove the equivalence it is enough to show that for any $R>0$ there exists $p_R>1$ such that
\begin{equation}\label{unioneDisgiunta}{\Omega}_{p,R}= \left({\Omega}_{p,R}\cap \mathcal N_p^+\right) \cup  \mathcal N_p^-\cup NL_p
\end{equation}
for $p\geq p_R$ and the union is disjoint.
Indeed \eqref{unioneDisgiunta} implies that 
\[\mathcal S_{p,R}=\max\left\{ \mathcal M_{p,R}, \ \frac{\|u_p^-\|_{L^{\infty}(\Omega)}^{p-1}}{\|u_p^+\|_{L^{\infty}(\Omega)}^{p-1}}\right\}\] and so the conclusion.
By definition $x_p^+\in \mathcal N_p^+$, moreover by \eqref{xppiuNonVedeLineaNodale} and \eqref{xppiuNonVedeboundary}
$\frac{d(x_p^+, NL_p)}{\mu_p^+}\rightarrow +\infty$ and $\frac{d(x_p^+, \partial\Omega)}{\mu_p^+}\rightarrow +\infty$ as $p\rightarrow +\infty$ and so for any $R>0$ there exists $p_R>1$ such that $\frac{d(x_p^+, NL_p)}{\mu_p^+}\geq 2R$ and $\frac{d(x_p^+, \partial\Omega)}{\mu_p^+}\geq 2R$ for $p\geq p_R$,  which implies that   $B_{R\mu_p^+}(x_p^+)\subset \mathcal N_p^+$ for $p\geq p_R$. As a consequence  $\left({\Omega}_{p,R}\cap \mathcal N_p^-\right)=\mathcal N_p^-, \ \left({\Omega}_{p,R}\cap NL_p\right)= NL_p$ 
and  the set $\left({\Omega}_{p,R}\cap \mathcal N_p^+\right)\neq \emptyset$ for $p\geq p_R$. Hence \eqref{unioneDisgiunta} follows from the fact that
\[{\Omega}_{p,R}=
\left({\Omega}_{p,R}\cap \mathcal N_p^+\right)
\cup \left({\Omega}_{p,R}\cap \mathcal N_p^-\right) \cup
\left({\Omega}_{p,R}\cap NL_p\right)
\] and the union is disjoint.
\end{proof}
\begin{remark}
Condition $(B1)$ can be equivalently written as
\[
\lim_{p\rightarrow +\infty}\frac{\mu_p^+}{\mu_p^-}=0.
\]
\end{remark}

\

\

\section{Asymptotic spectral analysis}\label{sectionAsymptoticAnaysis}

\subsection{Linearization of the limit problem}
In \cite{DicksteinPacellaSciunzi} the linearization at $U$ of the Liouville problem \eqref{LiouvilleProblem} has been studied. In the following we recall the main results.\\

For $v\in H^2(\R^2)$  we define the linearized operator by
\[
L^*(v):=-\Delta v -e^{U}v.
\]
We consider the Rayleigh functional $\mathcal R^*: H^1(\R^2)\rightarrow \R$
\[
\mathcal R^*(w):=\int_{\R^2}\left(|\nabla w|^2-e^{U}w^2 \right)\ dx
\]
and define 
\begin{equation}\label{minRayQuo}
\lambda_1^*:=\inf_{ \substack{ w\in H^1(\R^2)\\ \|w\|_{L^2(\R^2)}=1}}\mathcal R^*(w).
\end{equation}
 \begin{proposition} \label{prop:carattProbStar} We have the following
\begin{itemize}
\item[i)] $(-\infty< )\ \lambda_1^*<0$;
\item[ii)] every minimizing sequence of \eqref{minRayQuo} has a subsequence  strongly converging in $L^2(\R^2)$ to a minimizer;
\item[iii)] there exists a unique positive minimizer $\varphi_1^*$ to \eqref{minRayQuo}, which is radial and radially non-increasing;
\item[iv)]  $\lambda_1^*$ is an eigenvalue of $L^*$ and $\varphi_1^*$ is an eigenvector associated to $\lambda_1^*$. Moreover  $\varphi_1^*\in L^{\infty}(\R^2)$.
\end{itemize}
\end{proposition}
\begin{proof}
See \cite[Proposition 3.1]{DicksteinPacellaSciunzi}.
\end{proof}

\

\subsection{Linearization of the Lane-Emden problem}

In this section we consider the linearization of the Lane-Emden problem and study its connections with the linearization $L^*$ of the Liouville problem.

\

We define the linearized operator at $u_p$ of the Lane-Emden problem in $\Omega$ 
\[L_p(v):=-\Delta v-p |u_p|^{p-1}v,\ \ \
v\in H^2(\Omega)\cap H^1_0(\Omega)
.\]

Let $\lambda_{1,p}$ and $\varphi_{1,p}$ be respectively the first  eigenvalue and the first eigenfunction (normalized in $L^2$) of the operator $L_p$. It is well known that $\lambda_{1,p}< 0,$ $\forall\  p>1$ and that
$\varphi_{1,p}\geq 0$, moreover  $\|\varphi_{1,p}\|_{L^2(\Omega)}=1$.

\

Rescaling
\[\widetilde{\varphi}_{1,p}(x):= \mu_p^+ \ \varphi_{1,p}(x_p^++\mu_p^+ x), \ \ x\in \widetilde{\Omega}_p\]
and setting 
\[
\widetilde{\lambda}_{1,p}:= \left(\mu_p^+\right)^2\, \lambda_{1,p}
\]
then, it is easy to see that
$\widetilde{\varphi}_{1,p}$ and $\widetilde{\lambda}_{1,p}$ are respectively the first eigenfunction and first eigenvalue of the linear  operator $\widetilde{L}_p$ in $\widetilde{\Omega}_p$ with homogeneous Dirichlet boundary conditions, defined as
\[
\widetilde{L}_p v:=-\Delta v- V_p(x)\, v,\ \ \ \ \   
v\in H^2( \widetilde{\Omega}_p )\cap H^1_0( \widetilde{\Omega}_p ),
\]
where
\[
V_p(x):=\left|\frac{u_p(x_p^++\mu_p^+x)}{u_p(x_p^+)} \right|^{p-1}=\left|  1+\frac{v_p^+(x)}{p}  \right|^{p-1}
\ \ \ \ \ \mbox{($v_p^+$ is the function defined in Lemma \ref{convergenzaVp})}.\]
Observe that  $\|\widetilde{\varphi}_{1,p}\|_{L^2(\widetilde{\Omega}_p)}=\|\varphi_{1,p}\|_{L^2(\Omega)}=1$.

%
%
%
%
%
%
%
%

\

\begin{lemma}\label{lemmaFondamentalePotenziali}
Up to a subsequence
\[
\int_{\widetilde{\Omega}_p}\left(e^U- V_p\right)\widetilde{\varphi}_{1,p}^2\rightarrow 0 \ \mbox{ as }\ p\rightarrow +\infty.
\]

\end{lemma}
\begin{proof} For $R>0$ we divide the integral in the following way
\[
\int_{\widetilde{\Omega}_p}\left(e^U- V_p\right)\widetilde{\varphi}_{1,p}^2= 
\underbrace{
\int_{\widetilde{\Omega}_p\cap\{|x|\leq R\}}\left(e^U- V_p\right)\widetilde{\varphi}_{1,p}^2
}_{A_{p,R}}
 +
\underbrace{
\int_{\widetilde{\Omega}_p\cap\{|x|> R\}}\left(e^U- V_p\right)\widetilde{\varphi}_{1,p}^2
}_{B_{p,R}}.
\]
Now
\begin{equation}\label{step1A}
A_{p,R}=\int_{\widetilde{\Omega}_p\cap\{|x|\leq R\}}\left(e^U- V_p\right)\widetilde{\varphi}_{1,p}^2\rightarrow 0\ \ \mbox{ as }\ p\rightarrow +\infty, \ \mbox{ for all }\ R>0,
\end{equation}
since, up to a subsequence, $v_p^+\rightarrow U$ in $C^1_{loc}(\R^2)$ as $p\rightarrow +\infty$ (see Lemma \ref{convergenzaVp})
and so
$V_p\rightarrow e^U$ uniformly in $B_R(0)$, up to a subsequence.\\
On the other hand
\begin{eqnarray} 
B_{p,R} &=& \int_{\widetilde{\Omega}_p\cap\{|x|> R\}}\left(e^U- V_p\right)\widetilde{\varphi}_{1,p}^2 
=\underbrace{
 \int_{\widetilde{\Omega}_p\cap\{|x|> R\}}e^U \widetilde{\varphi}_{1,p}^2}_{C_{p,R}}-
\underbrace{
 \int_{\widetilde{\Omega}_p\cap\{|x|> R\}}V_p\widetilde{\varphi}_{1,p}^2}_{D_{p,R}}.
\end{eqnarray}
Following \cite{DicksteinPacellaSciunzi} we estimate
\[
(0\leq)\ \ C_{p,R}= \int_{\widetilde{\Omega}_p\cap\{|x|> R\}}e^U \widetilde{\varphi}_{1,p}^2\leq  \sup_{|x|>R}e^{U(x)} \int_{\widetilde{\Omega}_p}\widetilde{\varphi}_{1,p}^2=\sup_{|x|>R}e^{U(x)}=\sup_{|x|>R}\left( \frac{1}{1+\frac{|x|^2}{8}}\right)^2\leq \frac{64}{R^2}.
\]
Last we estimate
\begin{eqnarray*}
(0\leq)\ \ D_{p,R}=  \int_{\widetilde{\Omega}_p\cap\{|x|> R\}}V_p\widetilde{\varphi}_{1,p}^2
&\leq& \sup_{\widetilde{\Omega}_p\cap\{|x|> R\}} V_p(x)\ \int_{\widetilde{\Omega}_p}  \widetilde{\varphi}_{1,p}^2\\
&=&\sup_{\widetilde{\Omega}_p\cap\{|x|> R\}} V_p(x)\\
&=& \sup_{\widetilde{\Omega}_p\cap\{|x|> R\}}\left| \frac{u_p(x_p^+ +\mu_p^+ x)}{u_p(x_p^+)}\right|^{p-1}\\
&=&  \sup_{\Omega\cap\{|y-x_p^+|> R\mu_p^+\}}\left| \frac{u_p(y)}{u_p(x_p^+)}\right|^{p-1}\\
&=& \mathcal S_{p,R},
\end{eqnarray*}
so assumption $(B)$ implies that \[\lim_{R\rightarrow +\infty}\lim_{p\rightarrow +\infty} D_{p,R}=0,\]
and this concludes the proof.
\end{proof}

\

\begin{theorem}\label{theorem:convergenzaAutovalori}
We have, up to a subsequence, that
\begin{equation}
\widetilde{\lambda}_{1,p}\rightarrow \lambda_1^* \ \ \mbox{ as }\ p\rightarrow +\infty.
\end{equation}
\end{theorem}
\begin{proof} We divide the proof into two steps.\\
{\sl {\bf Step 1.} We show that, up to a subsequence, for $\epsilon >0$ 
\begin{equation}
\lambda_1^*\leq \widetilde{\lambda}_{1,p} + \epsilon \ \ \mbox{ for $p$ sufficiently large}.
\end{equation}
}
\begin{eqnarray}
\lambda_1^* &\leq& \int_{\R^2}\left( |\nabla \widetilde{\varphi}_{1,p}|^2-e^{U}\ \widetilde{\varphi}_{1,p}^2\right)\nonumber\\
& =& \int_{\widetilde{\Omega}_p} V_p\widetilde{\varphi}_{1,p}^2 + \widetilde{\lambda}_{1,p}\int_{\widetilde{\Omega}_p}\widetilde{\varphi}_{1,p}^2-\int_{\widetilde{\Omega}_p}e^{U}\ \widetilde{\varphi}_{1,p}^2\nonumber\\
& =& \widetilde{\lambda}_{1,p}- \int_{\widetilde{\Omega}_p}\left(e^U- V_p\right)\widetilde{\varphi}_{1,p}^2 
\end{eqnarray}
and so the conclusion follows by Lemma \ref{lemmaFondamentalePotenziali}.
\

\

{\sl {\bf Step 2.} We show that, up to a subsequence, for $\epsilon >0$ 
\begin{equation}
\widetilde{\lambda}_{1,p}\leq \lambda_1^* + \epsilon \ \ \mbox{ for $p$ sufficiently large}.
\end{equation}
}

The proof is similar to the one in \cite{DicksteinPacellaSciunzi}, we repeat it for completeness. For $R>0$, let us consider a cut-off regular function $\psi_R(x)=\psi_R(r)$ such that 
\[
\left\{
\begin{array}{lr}
0\leq\psi_R\leq1\\
\psi_R=1 \ \mbox{ for }\ r\leq R\\
\psi_R=0\ \mbox{ for }\ r\geq 2 R\\
|\nabla \psi_R|\leq 2/R
\end{array}
\right.
\]
and let us  set
\[ w_R:=\frac{\psi_R\varphi_1^*}{\|\psi_R\varphi_1^*\|_{L^2(\R^2)}}.\]
Hence, from the variational characterization of $\widetilde{\lambda}_{1,p}$ we deduce that
\begin{eqnarray}\label{pa}
\widetilde{\lambda}_{1,p} &\leq& \int_{\R^2}\left(|\nabla w_R|^2- V_p(x)w_R^2\right)dx\nonumber\\
&=& \int_{\R^2}\left(|\nabla w_R|^2- e^{U(x)}w_R^2\right)dx+ \int_{\R^2}\left(e^{U(x)}- V_p(x)\right)w_R^2dx.
\end{eqnarray}
Since $w_R\rightarrow \varphi_1^*$ in $H^1(\R^2)$ as $R\rightarrow +\infty$, it is easy to see that
given $\epsilon>0$ we can fix $R>0$ such that 
\begin{equation}\label{pe}
\int_{\R^2}\left(|\nabla w_R|^2- e^{U(x)}w_R^2\right)dx\leq \lambda_1^*+\frac{\epsilon}{2}.
\end{equation}
For such a fixed value of $R$ we can argue similarly as in the proof of \eqref{step1A} in Lemma \ref{lemmaFondamentalePotenziali} to obtain that, up to a subsequence in $p$,
\begin{equation}\label{pi}
\int_{\R^2}\left(e^{U(x)}- V_p(x)\right)w_R^2dx\leq \frac{\epsilon}{2}
\end{equation}
for $p$ large enough. Hence the proof of Step 2 follows from \eqref{pa}, \eqref{pe} and \eqref{pi}.
\end{proof}

\

\

\

\begin{corollary}\label{Cor:convergenzaAutofunzioni} Up to a subsequence
\[\widetilde{\varphi}_{1,p}\rightarrow \varphi_1^*\ \mbox{ in }\ L^2(\mathbb R^2)\ \mbox{ as }\ p\rightarrow +\infty.\]
\end{corollary}
\begin{proof}
By Lemma \ref{lemmaFondamentalePotenziali} and  Theorem \ref{theorem:convergenzaAutovalori} we have that, passing to a subsequence 
\[\lambda_{1}^*-\mathcal R^*(\widetilde{\varphi}_{1,p})= (\lambda_{1}^*-\widetilde{\lambda}_{1,p})+\widetilde{\lambda}_{1,p}-\mathcal R^*(\widetilde{\varphi}_{1,p})= (\lambda_{1}^*-\widetilde{\lambda}_{1,p})+
\int_{\R^2}\left(e^{U(x)}-V_p(x) \right)\widetilde{\varphi}_{1,p}^2\rightarrow 0\]
 as $p\rightarrow +\infty$, namely $\widetilde{\varphi}_{1,p}$ is a minimizing sequence for \eqref{minRayQuo} and so the result follows from points ii) and iii) of Proposition \ref{prop:carattProbStar}.
\end{proof}

%
%

\

\section{Proof of Theorem \ref{teorema:principale}} \label{sectionProofTheorem}
The proof of Theorem \ref{teorema:principale} follows the  same strategy as  in \cite{CazenaveDicksteinWeissler, MarinoPacellaSciunzi, DicksteinPacellaSciunzi} and it is a consequence of Proposition \ref{proposition:generaleCazenaveDickWeiss}, which is  a particular case of Theorem 2.3 in \cite{CazenaveDicksteinWeissler}.

Hence, to obtain Theorem \ref{teorema:principale}, it is enough to prove the following:
\begin{theorem}\label{condizioneProdottoScalarePositivo}
Let $u_p$ be a family sign changing solutions to \eqref{stationaryProblem} which satisfies conditions $(A)$ and $(B)$. Then there exists $p^*>1$ such that up to a subsequence, for $p>p^*$ 
\[\int_{\Omega}u_p\varphi_{1,p} >0,\]
where $\varphi_{1,p}$ is the first positive eigenfunction of the linearized operator $L_p$ at $u_p$.
\end{theorem}
\begin{proof}Since by an easy computation \[\int_{\Omega}u_p\varphi_{1,p}=\frac{p-1}{-\lambda_{1,p}}\int_{\Omega}|u_p|^{p-1}u_p\varphi_{1,p}\]
(see \cite[pg. 14]{DicksteinPacellaSciunzi}), we can study the sign of 
\[\int_{\Omega}|u_p|^{p-1}u_p\varphi_{1,p},\]
which is the same as the sign of
\[\frac{1}{u_p(x_p^+)^p\mu_p^+} \int_{\Omega}|u_p|^{p-1}u_p\varphi_{1,p}.\]

\

We show that, up to a subsequence,
\begin{equation}\label{conver}
\frac{1}{u_p(x_p^+)^p\mu_p^+}\int_{\Omega}|u_p|^{p-1}u_p\varphi_{1,p}\rightarrow \int_{\R^2}e^U \varphi_1^*\ (>0)\ \ \mbox{ as }\ p\rightarrow +\infty
\end{equation}
from which the conclusion follows.

\

In order to prove \eqref{conver} we change the variable and, for any $R>0$, we split the integral in the following way
\begin{eqnarray*}
\frac{1}{u_p(x_p^+)^p\mu_p^+}\int_{\Omega}|u_p|^{p-1}u_p\varphi_{1,p} &=& 
\frac{1}{u_p(x_p^+)^p}\int_{\widetilde{\Omega}_p}|u_p(x_p^++\mu_p^+x)|^{p-1}u_p(x_p^++\mu_p^+x)\widetilde{\varphi}_{1,p}(x)dx
\\
&=& \underbrace{\frac{1}{u_p(x_p^+)^p}\int_{\widetilde{\Omega}_p\cap\{|x|\leq R\}}|u_p(x_p^++\mu_p^+x)|^{p-1}u_p(x_p^++\mu_p^+x)\widetilde{\varphi}_{1,p}(x)dx}_{E_{p,R}}\\
&&
+ \underbrace{\frac{1}{u_p(x_p^+)^p}\int_{\widetilde{\Omega}_p\cap\{|x|> R\}}|u_p(x_p^++\mu_p^+x)|^{p-1}u_p(x_p^++\mu_p^+x)\widetilde{\varphi}_{1,p}(x)dx}_{F_{p,R}}
\end{eqnarray*}

By H\"older inequality, the convergence of $v_p^+$ to $U$ in $C^1_{loc}(\R^2)$ up to a subsequence (see Lemma \ref{convergenzaVp}) and Corollary \ref{Cor:convergenzaAutofunzioni} we have, for $R>0$ fixed:
\begin{eqnarray*}
\left|E_{p,R}-\int_{\{|x|\leq R\}}e^U\varphi_1^*\right| &\leq&   
\int_{\widetilde{\Omega}_p\cap\{|x|\leq R\}}\left|  \frac{u_p(x_p^++\mu_p^+x)}{u_p(x_p^+)}  \right |^{p}|\widetilde{\varphi}_{1,p}(x)-\varphi_1^*(x)|dx\\
&& + \int_{\{|x|\leq R\}}\varphi_1^*(x)\left|\left|1+\frac{v_p^+(x)}{p}    \right|^{p-1}\left( 1+\frac{v_p^+(x)}{p}  \right)  -e^U(x)\right|dx\\
&\leq & \|\widetilde{\varphi}_{1,p}-\varphi_1^*\|_{L^2(\R^2)}   \left[\int_{\widetilde{\Omega}_p\cap\{|x|\leq R\}}\left|  1+\frac{v_p^+(x)}{p}  \right |^{2p}\right]^{\frac{1}{2}}\\
&& +\sup_{\{|x|\leq R\}} \left|\left|1+\frac{v_p^+(x)}{p}    \right|^{p-1}\left( 1+\frac{v_p^+(x)}{p}  \right)  -e^U(x)\right| \int_{\{|x|\leq R\}}\varphi_1^*(x) dx\\
& \rightarrow & 0,
\end{eqnarray*}
as $p\rightarrow +\infty$, up to a subsequence.

For $R$ sufficiently large the term
\[
\int_{\{|x|> R\}}e^U\varphi_1^*dx
\]
may be made arbitrary small since $e^U\in L^1(\R^2)$ and $\varphi_1^*$ is bounded (Proposition \ref{prop:carattProbStar}- iv)).

\

\

Using H\"older inequality, $\|\widetilde{\varphi}_{1,p}\|_{L^2(\widetilde{\Omega}_p)}=1 $, assumption $(A)$ and \eqref{maxBoundedAwayFromZero} we have
\begin{eqnarray*}
|F_{p,R}|^2 
&\leq & \|\widetilde{\varphi}_{1,p}\|^2_{L^2(\widetilde{\Omega}_p)} 
\int_{\widetilde{\Omega}_p\cap\{|x|> R\}}\left|\frac{u_p(x_p^++\mu_p^+x)}{u_p(x_p^+)}\right|^{2p}dx
\\
&= & \int_{\widetilde{\Omega}_p\cap\{|x|> R\}}\left|\frac{u_p(x_p^++\mu_p^+x)}{u_p(x_p^+)}\right|^{2p}dx
\\
&= &\frac{1}{u_p(x_p^+)^2} \int_{\Omega\cap\{|y-x_p^+|>R\mu_p^+\}}\frac{p|u_p(y)|^{2p}}{u_p(x_p^+)^{p-1}}dy
\\
&\leq & \frac{p\int_{\Omega}|u_p|^{p+1}}{u_p(x_p^+)^2}
\left[ \sup_{\Omega\cap\{|y-x_p^+|>R\mu_p^+\}}\left| \frac{u_p(y)}{u_p(x_p^+)}\right|^{p-1}\ \right]
\\
& \stackrel{(A)+\eqref{maxBoundedAwayFromZero}}{ \leq} & C \mathcal S_{p,R}.
\end{eqnarray*}
And so by assumption $(B)$  we have
\[\lim_{R\rightarrow +\infty}\lim_{p\rightarrow +\infty}F_{p,R}=0.\]
\end{proof}

\

\section{A sufficient condition for $(B)$} \label{sectionSufficentConditionToB}
Next property is a sufficient condition for $(B)$:
\begin{lemma} \label{condizioneNecessariaGeneralissima} Assume that there exists $C>0$ such that
\begin{equation}\label{P_3^1}
|x-x_p^+|^2 p|\upp(x)|^{p-1}\leq C
\end{equation}
for all $p$ sufficiently large and all $x\in \Omega$. Then condition $(B)$ holds true up to a subsequence in $p$.
\end{lemma}

\begin{proof}
Let $R>0$ fixed and let $y\in\Omega$, $|y-x_p^+|> R\mu_p^+$, then for $p$ large, by \eqref{P_3^1}
\begin{eqnarray}
 \left|\frac{u_p(y)}{u_p(x_p^+)}\right|^{p-1}= \frac{p|u_p(y)|^{p-1}}{p|u_p(x_p^+)|^{p-1}}\stackrel{\eqref{P_3^1}}{\leq}
 \frac{C}{|y-x_p^+|^2}\frac{1}{p|u_p(x_p^+)|^{p-1}}\leq \frac{C}{R^2(\mu_p^+)^2}\frac{1}{p|u_p(x_p^+)|^{p-1}}=\frac{C}{R^2}
\end{eqnarray}
hence
\[0\leq\ \limsup_{p\rightarrow +\infty}\mathcal S_{p,R}\leq \frac{C}{R^2}\]
and $(B)$ follows, up to a subsequence in $p$, passing to the limit as $R\rightarrow +\infty$.
\end{proof}

\

\

Condition \eqref{P_3^1} is a special case of a more general result that  has been proved in \cite{DeMarchisIanniPacella2} 
 for any family $(\upp)$ of solutions to \eqref{stationaryProblem} under condition $(A)$ and which we recall here.

Given  $n\in\N\setminus\{0\}$ families of points $(\xip)$, $i=1,\ldots,n$  in $\Omega$ such that
\begin{equation}
p|\upp(\xip)|^{p-1}\to+\infty\ \mbox{ as }\ p\to+\infty,
\end{equation}
which implies in particular
\begin{equation}\label{seqmaggioridiuno}
\liminf_{p\rightarrow +\infty} u_p(x_{i,p})\geq 1,
\end{equation}
we define the parameters $\mip$ by
\bel
\mip^{-2}=p |\upp(\xip)|^{p-1},\ \mbox{ for all }\ i=1,\ldots,n,
\eel
and
we introduce the following properties:
\begin{itemize}
\item[$(\mathcal{P}_1^n)$] For any $i,j\in\{1,\ldots,n\}$, $i\neq j$,
\[
\lim_{p\to+\infty}\fr{|\xip-\xjp|}{\mip}=+\infty.
\]
\item[$(\mathcal{P}_2^n)$] For any $i=1,\ldots,n$,
\[
v_{i,p}(x):=\fr{p}{\upp(\xip)}(\upp(\xip+\mip x)-\upp(\xip))\To U(x)
\]
in $C^1_{loc}(\R^2)$ as $p\to+\infty$, where
\begin{equation}
U(x)=\log\left(\fr1{1+\fr18 |x|^2}\right)^2
\end{equation}
is the solution of $-\lap U=e^{U}$ in $\R^2$, $U\leq 0$, $U(0)=0$ and $\int_{\mathbb{R}^2}e^{U}=8\pi$.\\
Moreover
\begin{equation}\label{distDaBordo}\frac{d(x_{i,p},\partial\Omega)}{\mu_{i,p}}\rightarrow +\infty\quad\mbox{ and } \quad \frac{d(x_{i,p},NL_p)}{\mu_{i,p}}\rightarrow +\infty\ \mbox{ as }\ p\rightarrow +\infty.\end{equation}
\item[$(\mathcal{P}_3^n)$] There exists $C>0$ such that
\[
p \min_{i=1,\ldots,n} |x-\xip|^2 |\upp(x)|^{p-1}\leq C
\]
for all $p$ sufficiently large and all $x\in \Omega$.
\end{itemize}

\

\begin{proposition}[{\cite[Proposition 2.2]{DeMarchisIanniPacella2}}]
\label{prop2.2demarchisiannipacella2}
Let $(\upp)$ be a family of solutions to \eqref{stationaryProblem} and assume that $(A)$ holds. Then there exist $k\in\N\setminus\{0\}$ and $k$ families of points $(\xip)$ in $\Omega$  $i=1,\ldots, k$ such that $x_{1,p}=x_p^+$ and, after passing to a sequence, $(\mathcal{P}_1^k)$, $(\mathcal{P}_2^k)$, and $(\mathcal{P}_3^k)$ hold. Moreover, given any family of points $x_{k+1,p}$, it is impossible to extract a new sequence from the previous one such that $(\mathcal{P}_1^{k+1})$, $(\mathcal{P}_2^{k+1})$, and $(\mathcal{P}_3^{k+1})$ hold with the sequences $(\xip)$, $i=1,\ldots,k+1$. At last, we have
\[
\sqrt{p}\upp\to 0\quad\textrm{ in $C^1_{loc}(\bar\Omega\setminus\{\lim_p x_{i,p},\ i=1,\ldots, k\})\ $ as $p\to+\infty$.}
\]
\end{proposition}

\

\

 \begin{theorem}\label{corollaryCheUso}
Let $(\upp)$ be a family of solutions to \eqref{stationaryProblem} which satisfies condition $(A)$. Let $k\in\N\setminus\{0\}$ be 
the maximal number  of families of points $(x_{i,p})$, $i=1,\ldots, k$, for which $(P^k_1)$, $(P^k_2)$ and  $(P^k_3)$ hold up to a subsequence (as in Proposition \ref{prop2.2demarchisiannipacella2}).

If $k=1$ then condition $(B)$ holds true up to a subsequence in $p$.
\end{theorem}
\begin{proof}
By $(A)$ Proposition \ref{prop2.2demarchisiannipacella2} applies. Just observe that $x_{1,p}=x_p^+$ and so when $k=1$ property $(\mathcal P_3^k)$ reduces to \eqref{P_3^1} and so the conclusion follows by Lemma \ref{condizioneNecessariaGeneralissima}.
\end{proof}

\

\

In the following we  use Theorem \ref{corollaryCheUso} to  obtain condition $(B)$ for suitable classes of solutions.

\

\

\section{Proof of Theorem \ref{teoremaExistenceAB}}\label{sectionCasoGSimmetrico}

Before proving Theorem \ref{teoremaExistenceAB} we observe that the existence of sign changing stationary solutions $u_p$ to \eqref{problem} satisfying  assumptions \eqref{assumptionGruppoSimmetria} and \eqref{energiaLimitata} has been proved for $m\geq 4$ in \cite{DeMarchisIanniPacella1} for  $p$ large.  The proof uses the fact that the energy is decreasing along non constant solutions, and relies on constructing a suitable initial condition $v_0$ for problem \eqref{problem} such that any stationary solution in the corresponding $\omega$-limit set satisfies the energy estimate \eqref{energiaLimitata}. This construction can be done for $p$ large even without any symmetry assumption on $\Omega$ (see \cite{DeMarchisIanniPacella1} for details).
Anyway when $\Omega$ is a simply connected $G$-symmetric smooth bounded domain with $|G|\geq m$ also some qualitative properties of $u_p$ under condition \eqref{energiaLimitata} may be obtained  (for instance the nodal line does not touch $\partial\Omega$, it does not pass through the origin, etc, as shown in \cite{DeMarchisIanniPacella1}).

Then, in \cite{DeMarchisIanniPacella2} a deeper asymptotic analysis of $u_p$ as $p\rightarrow +\infty$  has been done, showing concentration in the origin and  a bubble tower behavior, when $\Omega$ is a simply connected $G$-symmetric smooth bounded domain with $|G|\geq m e$. \\

Here we do not require $\Omega$ to be simply connected.\\

Clearly assumption \eqref{energiaLimitata} is a special case of condition $(A)$, hence in particular Proposition \ref{prop2.2demarchisiannipacella2} holds.
As before we assume w.l.o.g. that $\|u_p\|_{\infty}=\|u_p^+\|_{\infty}$.\\
\\
The proof of Theorem \ref{teoremaExistenceAB} follows then from the following  
\begin{proposition}\label{propValeBcasoG}
Let $\Omega\subset\R^2$ be a
smooth bounded domain, $O\in\Omega$,  invariant under the action of a cyclic group $G$ of
rotations about the origin which  satisfies \eqref{assumptionGruppoSimmetria} for a certain $m>0$. Let $(\upp)$ be a family of sign changing $G$-symmetric  stationary solutions of \eqref{problem} which satisfies
\eqref{energiaLimitata}.
 Then
condition $(B)$ is satisfied up to a subsequence.
\end{proposition}

As we will see Proposition \ref{propValeBcasoG} is a consequence of the general sufficient condition in Theorem \ref{corollaryCheUso}.

Hence in order to prove it  we only need to show that $k=1$,
where the number $k$  is  the maximal number  of families of points $(x_{i,p})$, $i=1,\ldots, k$, for which $(P^k_1)$, $(P^k_2)$ and  $(P^k_3)$ hold, up to a subsequence, as in Proposition \ref{prop2.2demarchisiannipacella2}. 
When $m=4$ the result has been already proved in \cite[Proposition 3.6]{DeMarchisIanniPacella2}. Here we show the general case (see also \cite[Remark 4.6]{DeMarchisIanniPacella2}). We start with  the following:
\begin{lemma} \label{lemma:proposition3.3}
Let $\Omega\subset\R^2$ be a
smooth bounded domain, $O\in\Omega$,  invariant under the action of a cyclic group $G$ of
rotations about the origin which  satisfies \eqref{assumptionGruppoSimmetria} for a certain $m>0$. Let $(\upp)$ be a family of sign changing $G$-symmetric  stationary solutions of \eqref{problem} which satisfies
\eqref{energiaLimitata}.

Let  $k, x_{i,p}$ and $\mu_{i,p}$ for $i=1,\ldots, k$ be as in  Proposition \ref{prop2.2demarchisiannipacella2}.
Then 
\[
\frac{|x_{i,p}|}{\mu_{i,p}} \ \mbox{ is bounded.}
\]
\end{lemma}

\begin{proof}
The proof is similar to the one of \cite[Proposition 3.3]{DeMarchisIanniPacella2}.\\
Let us fix $i\in\{1,\ldots, k\}$. In order to simplify the notation we drop the dependence on $i$ namely we set
$x_{p}:=x_{i,p}$ and $\mu_{p}:=\mu_{i,p}$.\\

Without loss of generality we can assume that either $(x_{p})_p\subset \mathcal{N}_p^{+}$ or $(x_{p})_p\subset \mathcal{N}_p^{-}$. We prove the result in the case $(x_{p})_p\subset \mathcal{N}_p^{+}$, the other case being similar.\\ 

Let $h:=|G|$, ($\mathbb N\setminus\{0\}\ni h\geq me$) and let us denote by $g^j$, $j=0,\dots, h-1$, the elements of $G$. We consider the rescaled nodal domains
\[
\widetilde{\mathcal N_{p}^+}^{j} :=\{x\in\mathbb R^2\ : \ \mu_p x +g^jx_p\in \mathcal N_p^+\}
,\ \ j=0,\dots, h-1,\]
 and the rescaled functions $z_{p}^{j,+}(x): \widetilde{\mathcal N_{p}^+}^{j} \rightarrow\R$ defined by
\begin{equation}\label{z_j} z_{p}^{j,+}(x):=\frac{p}{u_{p}^+(x_{p})}\left( u_{p}^+(\mu_{p} x+g^jx_{p})-u_{p}^+(x_{p}) \right), \ \ j=0,\dots, h-1.\end{equation}

Observe that, since $\Omega$ is $G$-invariant, $g^j x_p\in\Omega$ for any $j=0,\dots, h-1$. Moreover  $u_{p}$ is $G$-symmetric and $x_p$ satisfies \eqref{distDaBordo}, hence it's not difficult to see from $(\mathcal P_2^{k})$ that each function $z_{p}^{j,+}$ converges to $U$ in $C^1_{loc}(\mathbb R^2)$, as $p\rightarrow \infty$ and $8\pi =\int_{\R^2}e^{U}dx$
(see also \cite[Corollary 2.4]{DeMarchisIanniPacella2}).\\

Assume by contradiction that there exists a sequence $p_n\rightarrow +\infty$ such that $\frac{|x_{p_n}|}{\mu_{p_n}}\rightarrow + \infty$. Let  
\[d_n:=|g^j x_{p_n}-g^{j+1}x_{p_n}|,\quad j=0,..,h-1.\]
Then, since the $h$ distinct points $g^j x_{p_n}$, $j=0,\ldots, h-1$, are the vertices of a regular polygon centered in $O$,   $d_n=2\widetilde d_n \sin{\frac{\pi}{h}}$, where $\widetilde d_n:=|g^jx_{p_n}|\equiv |x_{p_n}|$, $j=0,..,h-1$. Hence 
\begin{equation}\label{denneinfinito}
\frac{d_n}{\mu_{p_n}}\rightarrow +\infty.
\end{equation}
Let \begin{equation}\label{R_n}R_{n}:=\min\left\{\frac{d_n}{3},\frac{d(x_{p_n},\partial\Omega)}{2}, \frac{d(x_{p_n}, NL_{p_n})}{2}\right\},
\end{equation}
then  by \eqref{denneinfinito} and \eqref{distDaBordo} 
\begin{equation}\label{invadeR2}
\frac{R_n}{\mu_{p_n}}\rightarrow  +\infty,
\end{equation}
moreover, by construction,  
\begin{eqnarray}
& B_{R_n}(g^j x_{p_n})\subseteq \mathcal{N}_{p_n}^+, \ \ \mbox{ for }\ j=0,\dots,h-1 \label{contenutoInOmega} \\
&B_{R_n}(g^j x_{p_n})\cap B_{R_n}(g^l x_{p_n}) =\emptyset,\ \ \mbox{ for }j\neq l.  \label{palleDisgiunte} 
\end{eqnarray}
 Using \eqref{invadeR2}, the convergence of $z_{p_n}^{j,+}$ to $U$,  \eqref{seqmaggioridiuno} and Fatou's lemma, we have
\begin{eqnarray}\label{betterEstimate}
8\pi &=&\int_{\R^2}e^{U}dx\nonumber
\\
&\stackrel{\textrm{Fatou + conv. of $v_{p_n}^j$} + \eqref{invadeR2}}{\leq}& \lim_n \int_{ B_{\frac{R_n}{\mu_{p_n}}}(0)  } e^{z_{p_n}^{j,+} + (p_n+1)\left(\log{\left|1+\frac{z_{p_n}^{j,+}}{p_n}\right|}-\frac{z_{p_n}^{j,+}}{(p_n+1)}\right)}dx\nonumber
\\
&= &\lim_n \int_{B_{\frac{R_n}{\mu_{p_n}}}(0)}\left|1+\frac{z_{p_n}^{j,+}(x)}{p_n} \right|^{(p_n+1)}dx\nonumber
\\
&=& \lim_n \int_{B_{\frac{R_n}{\mu_{p_n}}}(0)}\left|\frac{ u^+_{p_n}(\mu_{p_n} x+g^jx_{p_n})}{ u^+_{p_n}(x_{p_n})}dx     \right|^{(p_n+1)}dx\nonumber
\\
&=& \lim_n \int_{B_{R_n}(g^jx_{p_n})}\frac{\left| u^+_{p_n}(x)\right|^{(p_n+1)}}{(\mu_{p_n})^2 \left|u^+_{p_n}(x_{p_n})\right|^{(p_n+1)}}dx\nonumber
\\
&=&\lim_n \frac{p_n}{\left|u^+_{p_n}(x_{p_n})\right|^2}  \int_{B_{R_n}(g^jx_{p_n})} \left| u^+_{p_n}(x)\right|^{(p_n+1)}dx\nonumber
\\
&\stackrel{\eqref{seqmaggioridiuno}}{\leq}&
 \lim_n p_n\int_{B_{R_n}(g^jx_{p_n})} \left| u^+_{p_n}(x)\right|^{(p_n+1)}dx.
\end{eqnarray}
Summing on $j=0,\dots, h-1$, using \eqref{palleDisgiunte}, \eqref{contenutoInOmega} and assumption \eqref{energiaLimitata} we get:
\begin{eqnarray*}
h\cdot 8\pi
&\leq &
\lim_n\ p_n \sum_{j=0}^{h-1} \int_{B_{R_n}(g^jx_{p_n})} \left| u^+_{p_n}(x)\right|^{(p_n+1)}dx
\\
&\stackrel{\eqref{palleDisgiunte} + \eqref{contenutoInOmega} }{\leq}&
\lim_n\   p_n\int_{\mathcal{N}_{p_n}^+} \left| u_{p_n}(x)\right|^{(p_n+1)}dx
\\
&=&
\lim_n\ \left(  p_n\int_{\Omega} \left| u_{p_n}(x)\right|^{(p_n+1)}dx  -  p_n\int_{\mathcal{N}_{p_n}^-} \left| u_{p_n}(x)\right|^{(p_n+1)}dx \right)
\\
&\stackrel{{\footnotesize\mbox{\cite[Lemma 3.1]{DeMarchisIanniPacella2}}}}{\leq} &
\lim_n\   p_n\int_{\Omega} \left| u_{p_n}(x)\right|^{(p_n+1)}dx  -\   8\pi e 
\\
&\stackrel{\eqref{energiaLimitata}}\leq & (\alpha -1)\ 8\pi e
\\
&< & m\ 8\pi e
\end{eqnarray*}
  which gives a  contradiction with \eqref{assumptionGruppoSimmetria}.
  \end{proof}
  
  \

Last using Lemma \ref{lemma:proposition3.3} we can  prove that the number $k$  in  Proposition \ref{prop2.2demarchisiannipacella2}
 is equal to one. 
\begin{lemma}\label{lemma:k=1Radiale}
Let $\Omega\subset\R^2$ be a
smooth bounded domain, $O\in\Omega$,  invariant under the action of a cyclic group $G$ of
rotations about the origin which  satisfies \eqref{assumptionGruppoSimmetria} for a certain $m>0$. Let $(\upp)$ be a family of sign changing $G$-symmetric  stationary solutions of \eqref{problem} which satisfies
\eqref{energiaLimitata}.
Let  $k$  be, as in  Proposition \ref{prop2.2demarchisiannipacella2}, the maximal number  of families of points $(x_{i,p})$, $i=1,\ldots, k$, for which, after passing to a subsequence, $(P^k_1)$, $(P^k_2)$ and  $(P^k_3)$ hold. Then
\[k=1.\]
\end{lemma}
\begin{proof} The proof is the same as in  \cite[Proposition 3.6]{DeMarchisIanniPacella2}, we repeat it for completeness.
Let us assume by contradiction that $k > 1$ and set $x^+_p=x_{1,p}$. For a family $(x_{j,p})$, $j\in\{2,\ldots, k\}$ by  Lemma \ref{lemma:proposition3.3}, there exists $C>0$ such that
\[
\frac{|x_{1,p}|}{\mu_{1,p}}\leq C\quad\textrm{and}\quad\frac{|x_{j,p}|}{\mu_{j,p}}\leq C.
\]
Thus, since by definition $\mu^+_p=\mu_{1,p}\leq \mu_{j,p}$, also
\[
\frac{|x_{1,p}|}{\mu_{j,p}}\leq C.
\]
Hence
\[
\frac{|x_{1,p}-x_{j,p}|}{\mu_{j,p}}\leq\frac{|x_{1,p}|+|x_{j,p}|}{\mu_{j,p}}\leq C\quad \textrm{as $p\to+\infty$},
\]
which  contradicts $(\mathcal{P}_1^k)$.
\end{proof}
\

\

\section{A special case in Theorem \ref{teoremaExistenceAB}: the radial solutions} \label{sectionCasoRadiale}

In this section we show that, when the domain $\Omega$ is the unit ball in $\R^2$, the unique (up to a sign) radial solution 
$u_{p,\mathcal K}$ of \eqref{stationaryProblem} with  $\mathcal K\geq 2$ nodal  regions satisfies conditions $(A)$ and $(B)$.\\

Thus Theorem \ref{teorema:principale} applies to $u_{p,\mathcal{K}}$, namely we re-obtain the result already known from \cite{DicksteinPacellaSciunzi} through a different proof which does not rely on radial arguments.\\

Let us fix the number of nodal regions $\mathcal K\geq 2$.  As before we  assume w.l.o.g. that $\|u_{p,\mathcal K}\|_{\infty}=\|u_{p,\mathcal K}^+\|_{\infty}$. \\

The main result is the following:
\begin{proposition}
Let $\Omega$ be the unit ball in $\R^2$ and for $\mathcal K\geq 2$ let $u_{p,\mathcal K}$ be the unique radial solution of \eqref{stationaryProblem} with $\mathcal K$ nodal domains. 
Then there exists $m (=m(\mathcal K))$ $>0$ for which the assumptions
of Theorem \ref{teoremaExistenceAB} are satisfied.
\end{proposition}
\begin{proof}
In \cite[Proposition 2.1]{DicksteinPacellaSciunzi} it has been proved that $u_{p,\mathcal{K}}$ satisfies assumption $(A)$ (by extending the arguments employed in \cite{DeMarchisIanniPacella1} for the case with two nodal regions). Hence there exists $C(=C(\mathcal{K}))$ such that
\[p\int_{\Omega}|u_{p,\mathcal{K}}|^{p+1}dx\leq C.\]
Let us define 
\[\alpha:=\frac{C}{8\pi e}\]
and let  $\bar m>0$ be such that \[\bar m>\frac{C}{8\pi e} -1.\]
Let $G$ be a cyclic group of rotations about the origin such that  $|G|\geq \bar me$.
Of course the unit ball is $G$ invariant, moreover, since $u_{p,\mathcal K}$ is radial, it is in particular $G$-symmetric and so we have proved that  \eqref{assumptionGruppoSimmetria} and \eqref{energiaLimitata} hold true with $m=\bar m$. 
\end{proof}

\

We conclude the section with some more consideration on condition $(B)$ in the radial case.

\

It is easy to show that  (see \cite[Proposition 2.4 - i)]{DicksteinPacellaSciunzi} for the proof) $\|u_{p,\mathcal K}\|_{\infty}=u_{p,\mathcal K}(0)\ (>0)$, namely $x_p^+\equiv 0$. Hence condition $(B)$ in this radial case reads as follows:
\begin{equation}  \tag{B}
\lim_{R\rightarrow +\infty}\lim_{p\rightarrow +\infty} S_{p,R}=0,
\end{equation}
where, for $R>0$, 
\[
S_{p,R}:=\sup \left\{ \left|\frac{u_{p,\mathcal K}(r)}{u_{p,\mathcal K}(0)}\right|^{p-1}\ : \  \ R\mu_p^+<r<1 \right\}
\]
and $u_{p,\mathcal K}(r)=u_{p,\mathcal K}(|x|)$, $r=|x|$.
\\
\\
\\
In addition to the general characterization in Proposition \ref{proposition:caratterizzazioneB}, it is easy to prove in the radial case, also the following characterization of condition $(B)$:
 \begin{proposition}
Let $\Omega$ be the unit ball in $\R^2$ and for $\mathcal K\geq 2$ let $u_{p,\mathcal K}$ be the unique radial solution of \eqref{stationaryProblem} with $\mathcal K$ nodal domains.
Set $0< r_{p,1} <r_{p,2}<\ldots <r_{p,\mathcal K-1} <1 $ the nodal radii of $u_{p,\mathcal K}(r)$. \\
Then for any $R>0$ there exists $p_R>1$ such that the set $\{  R\mu_p^+<r<r_{p,1} \}\neq \emptyset$ for $p\geq p_R$ and so
\begin{equation} \label{definitionMpRprimo}
\mathcal M'_{p,R}:=\sup \left\{ \left|\frac{u_{p,\mathcal K}(r)}{u_{p,\mathcal K}(0)}\right|^{p-1}\ : \ R\mu_p^+<r< r_{p,1} \right\}
\end{equation} 
is well defined.\\
Moreover condition $(B)$ is equivalent to 
\[
 \left\{
 \begin{array}{lr}
 \displaystyle{ \lim_{p\rightarrow +\infty}}\sup_{\{r_{p,1}<r<1\} }\frac{ |u_{p,\mathcal K}(r)|^{p-1}}{u_p(0)^{p-1}}= 0  & \qquad (B1')\\
 \\
 \displaystyle{\lim_{R\rightarrow +\infty}}\lim_{p\rightarrow +\infty}  \mathcal M'_{p,R}=0 & \qquad (B2')
 \end{array}
 \right.
 \]
 \end{proposition} 
\begin{remark}
Observe that $(B1')$ was already known, indeed  in \cite[Proposition 2.4 - iv)]{DicksteinPacellaSciunzi} the authors proved that
\[\sup_{\{r_{p,1}<r<1\} }\frac{ |u_{p,\mathcal K}(r)|}{u_{p,\mathcal K}(0)}\longrightarrow \alpha <\frac{1}{2}\ \mbox{ as }\ p\rightarrow +\infty.\] 
\end{remark}

\

\

\end{document}